\documentclass{amsart}
\usepackage{fullpage}
\usepackage{epsfig, amsmath,xspace}
\usepackage{amssymb,amscd}
\usepackage[all,cmtip]{xy}
\usepackage{hyperref}

\newtheorem{thm}{Theorem}[section]
\newtheorem{lem}[thm]{Lemma}
\newtheorem{prop}[thm]{Proposition}
\newtheorem{cor}[thm]{Corollary}
\theoremstyle{definition}
\newtheorem{defn}[thm]{Definition}

\newtheorem{ex}[thm]{Example}
\newtheorem{nota}[thm]{Notation}

\theoremstyle{remark}
\newtheorem{rem}[thm]{Remark}
\numberwithin{equation}{section}

\newdir{ >}{{}*!/-7pt/@{>}}
\newcommand{\dgmod}{{\sf dgmod}}

\newcommand{\Hom}{\mathrm{Hom}}
\newcommand{\End}{\mathrm{End}}
\newcommand{\Ker}{\mathrm{Ker}}
\newcommand{\id}{\mathrm{id}}
\newcommand{\im}{\mathrm{Im}}
\def\cal{\mathcal}

\def\O{\mathcal{O}}

\def\Z{\mathbb{Z}}
\def\N{\mathbb{N}}

\def\dd{\mathrm{d}}

\def\id{\mathrm{id}}
\def\leq{\leqslant}
\def\geq{\geqslant}

\def\rto{\rightarrow}

\newcommand{\epi}{\twoheadrightarrow}

\begin{document}
\title[Pre-Lie systems and obstruction to $A_\infty$-structures over a ring]{Pre-Lie systems and obstruction to $A_\infty$-structures over a ring}
\author{Muriel Livernet}
\address{Universit\'e Paris 13, CNRS, UMR 7539 LAGA, 99 avenue
  Jean-Baptiste Cl\'ement, 93430 
  Villetaneuse, France}
\email{livernet@math.univ-paris13.fr}
\keywords{Obstruction theory, $A_\infty$-structures, pre-Lie systems}
\subjclass[2000]{16E45, 16E40, 55G35}
\date{\today}
\begin{abstract}The aim of the note is to prove an obstruction theorem for $A_\infty$-structures over a commutative ring $R$. Given a $\Z$-gra\-ded $A_m$-algebra, with $m\geq 3$, we give conditions on the  Hochschild cohomology of the associative algebra $H(A)$ so that the $A_{m-1}$-structure can be lifted to an $A_{m+1}$-structure. 
These conditions apply in case we start with an associative algebra up to homotopy and want to lift this structure to an $A_\infty$-structure. The hidden purpose of the note is to show
that there are no assumptions needed on the commutative ring $R$ nor bounded assumptions on the complex $A$.
\end{abstract}
\maketitle

\section*{Introduction}

The purpose of this note is to fill a gap in the litterature concerning $A_\infty$-structures in the category of differential $\Z$-graded $R$-modules when $R$ is a commutative ring. We are concerned with obstruction theory for the existence of an $A_\infty$-structure on a dgmodule $V$ endowed with a product which is associative up to homotopy. We answer the question of the existence of higher homotopies, in terms of the Hochschild cohomology of the associative algebra $H_*(V)$. In the context of $A_\infty$-spaces or $A_\infty$-spectra this question has been answered by Robinson in \cite{Rob89}. If one applies the zig-zag of equivalences   between the category of modules over the Eilenberg Mac Lane ring spectrum $HR$ and the category of differential graded $R$-modules described by Shipley in \cite{Shipley07}, one gets the result we want. Our purpose is to give a direct account of the method in the differential graded context.  Note that the question has also been studied by Lef\`evre-Hasegawa for minimal $A_\infty$-algebras on a field in \cite{Lefevre03}. We follow the lines of his approach.

We recall that $A_\infty$-structures were defined by Stasheff in \cite{Stasheff63} for spaces, in order to give a recognition principle for loop spaces. In order to do so, he built an operad based on  associahedra. The simplicial chain complex of this operad is
what is known as the $A_\infty$-operad in the category of differential graded modules. Algebras over this operad are called $A_\infty$-algebras. Kadesishvili studied in \cite{Kadei80} an obstruction theory for the uniqueness of $A_\infty$-structures, also in terms of the Hochschild
cohomology of $H_*(A)$, when $A$ is an $A_\infty$-algebra. In this note we study the existence rather than the uniqueness of such a structure.

In the process of building an obstruction theory for the existence of $A_\infty$-structures on a ring $R$ we encountered two assumptions commonly needed on $R$-modules.
The first one is the assumption that every $R$-module considered should have no $2$-torsion. We discovered that this hypothesis is not needed, if one takes a closer look at the Lie algebra structure usually used to solve obstruction issues. Indeed, in our context, the Lie algebra structure not only comes from a pre-Lie algebra structure but from a pre-Lie system as defined by Gerstenhaber in \cite{Ger63}. The first section of the note is concerned with pre-Lie systems.
The second  assumption is that every graded module over $R$ should be $\N$-graded. Again this hypothesis is not needed, and we prove in the second section that, under some projectivity conditions, there is an isomorphism
between $H(\Hom(C,D))$ and $\Hom(H(C),H(D))$ for any differential $\Z$-graded $R$-modules $C$ and $D$.
The last section is devoted to obstruction theory.

\medskip

\noindent{\bf Notation.} We work over a commutative ring $R$. 

We denote by $\dgmod$ the category of lower $\Z$-gra\-ded $R$-modules with a differential of degree $-1$. Objects in this category are called dgmodules for short. 
\begin{itemize}
\item
The category $\dgmod$ is symmetric monoidal for the tensor product
$$(C\otimes D)_n=\oplus_{i+j=n} C_i\otimes_R D_j$$ with the differential given by
$$\partial(c_i\otimes d_j)=\partial_C(c_i)\otimes d_j+(-1)^i c_i\otimes\partial_D(d_j),\quad\forall c_i\in C_i, d_j\in D_j.$$
\item For $C$ a dgmodule, we denote by $sC$ its suspension, that is, $(sC)_i=C_{i-1}$ with differential
$\partial(sc)=-s\partial_C(c)$. 
\item Let $C$ be a dgmodule. For $c\in C_i$ and $d\in C_j$,  we will use the notation $\epsilon_{c,d}=(-1)^{ij}.$
\item Let $C$ and $D$ be dgmodules. We denote by $\Hom(C,D)$ the dgmodule
$$\Hom_i(C,D)=\prod_n \Hom_R(V_n,W_{n+i})$$
with differential
$\partial: \Hom_i(C,D)\rightarrow\Hom_{i-1} (C,D)$
defined for $c\in C_n$ by
$$(\partial f)_n(c)=\partial_D(f_n(c))-(-1)^i f_{n-1}(\partial_C c).$$
\item We use the Koszul sign rule: let $C,C',D$ and $D'$ be dgmodules; for $f\in \Hom(C,D)$ and $g\in \Hom(C',D')$ the map $f\otimes g\in \Hom(C\otimes C',D\otimes D')$ is defined by
$$\forall x\in C\otimes C', (f\otimes g)(x\otimes y)=\epsilon_{x,g} f(x)\otimes g(y).$$
\item The suspension map $s: C\rto (sC)$ has degree $+1$. The Koszul sign rule implies that
$$(s^{-1})^{\otimes n}\circ s^{\otimes n}=(-1)^{\frac{n(n-1)}{2}}\; \id_{C^{\otimes n}}.$$

\end{itemize}

\medskip

\noindent{\bf Aknowledgment.} I am indebted to Benoit Fresse, Birgit Richter and Sarah Whitehouse for valuable discussions.

\section{pre-Lie systems and graded pre-Lie algebras} \label{sec:preLie}

 Pre-Lie systems and pre-Lie algebras have been introduced by M. Gerstenhaber in \cite{Ger63}, in order to understand the richer algebra structure on the complex computing the Hochschild cohomology of an associative algebra $A$, yielding to the ``Gerstenhaber structure" on the Hochschild cohomology of $A$. In this section we review some of the results of Gerstenhaber, together with variations on the gradings and signs involved. Namely, different pre-Lie structures, as in Proposition  \ref{P:unsignedcirc}  and in Theorem \ref{P:end}, are described from a given pre-Lie system, depending on the grading we choose.
 
The main result of the section is the technical Lemma \ref{L:preLie_formula}, allowing to use pre-Lie systems on a ring with no assumptions concerning the $2$-torsion. It is one of the key ingredient in the proof of the obstruction Theorem \ref{T:main}.

Throughout the section we are given a $(\Z,\N)$-bigraded $R$-module $\bigoplus\limits_{n\in\N,i\in\Z}{\O}_i^n$. The examples we have in mind are  
\begin{itemize}
\item $\End^n_i(V)=\Hom_i(V^{\otimes n},V)$, for a dgmodule $V$.
\item More generally ${\O(n)}_i$, for an operad  $\O$, symmetric or not, see {\it e.g.} \cite{KapMan01}.
\item $\Hom(\cal C,\cal P)$ for given (non-symmetric) cooperad $\cal C$ and operad $\cal P$ or $\Hom_{\mathbb S}(\cal C,\cal P)$ for a cooperad $\cal C$ and an operad $\cal P$ where $\Hom_{\mathbb S}$ is the subset of $\Hom$ of invariant maps under the action of the symmetric group. This example is an application of the previous one, since $\Hom(\cal C,\cal P)$ forms an operad, the {\it convolution operad}, as defined by Berger and Moerdijk in \cite{BerMoe03}.
\end{itemize}
The paper will deal with the first item. It may be understood as the toy model for obstruction theory for $\O_\infty$-algebras, where
$\O$ is a Koszul operad. If one would like to extend the result of this paper for operads, one would use the third item, as suggested in the book in progress of Loday and Vallette \cite{LodVal}.

\begin{nota} Let $\cal O$ be a $(\Z,\N)$-bigraded $R$-module. For $x\in \cal O^n_i$, the integer $n$ is called the {\sl arity} of $x$, the integer $i$ is called the {\sl degree} of $x$ and the integer $i+n-1$ is called the {\sl weight} of $x$ and denoted by $|x|$. 
For a fixed $n$, we consider the $\Z$-graded $R$-module $\cal O^n=\oplus_i \cal O^n_i$.
\end{nota}

\subsection{General case}

\begin{defn}\label{D:preLiesystem} Let $\mathcal O$ be a $(\Z,\N)$-bigraded $R$-module. A {\it graded pre-Lie system} on $\O$ is a sequence of maps, called composition maps,
$$\circ_k: \O^n_i\otimes \O^m_j \rightarrow \O^{n+m-1}_{i+j},\ \forall 1\leq k\leq n,$$
satisfying the relations: for every $f\in {\O}^n_i, g\in {\O}^m_j$ and $h\in {\O}^p_l$ 
\begin{eqnarray*}
 f\circ_u (g\circ_v h)&=&(f\circ_u g)\circ_{v+u-1} h,  \ \ \forall 1\leq u\leq n \text{ and } 1\leq v\leq m, \\
(f\circ_u g)\circ_{v+m-1} h&=&(-1)^{jl}(f\circ_v h)\circ _u g, \ \  \forall 1\leq u<v\leq n. 
\end{eqnarray*}
We will denote by $(\cal O,\circ)$ a graded pre-Lie system.
\end{defn}

\begin{defn}\label{D:wpreLiesystem} Let $\mathcal O$ be a $(\Z,\N)$-bigraded $R$-module. A {\it weight graded pre-Lie system} on $\O$ is a sequence of maps, called composition maps,
$$\circ_k: \O^n_i\otimes \O^m_j \rightarrow \O^{n+m-1}_{i+j},\ \forall 1\leq k\leq n$$
satisfying the relations: for every $f\in {\O}^n_i, g\in {\O}^m_j$ and $h\in {\O}^p_l$ 
\begin{eqnarray} \label{E:syst1}
 f\circ_u (g\circ_v h)&=&(f\circ_u g)\circ_{v+u-1} h,\ \ \forall 1\leq u\leq n \text{ and } 1\leq v\leq m, \\
(f\circ_u g)\circ_{v+m-1} h&=&(-1)^{(j+m-1)(l+p-1)}(f\circ_v h)\circ _u g,\ \  \forall 1\leq u<v\leq n. \label{E:syst2}
\end{eqnarray}

Note that the composition maps preserve the weight grading.
\end{defn}

A short computation proves the following Proposition.

\begin{prop}\label{P:graded2weight} Any graded pre-Lie system gives rise to a weight graded pre-Lie system and vice versa. Namely, 
if $(\cal O,\star)$ is a graded pre-Lie system, then the collection
$$\circ_k :\O^n_i\otimes \O^m_j \rightarrow \O^{n+m-1}_{i+j},\ \forall 1\leq k\leq n$$ defined by
$$f \circ_k g=(-1)^{(j+m-1)(n-1)+(m-1)(k-1)} f\star_k g$$
is a weight graded pre-Lie system.
\end{prop}

\begin{ex}\label{E:fund} Given a graded operad $\cal O$, the definition of the axioms for partial composition coincides with the one for pre-Lie systems. Hence the collection $\O^n_i=\O(n)_i$ forms a pre-Lie system. 
In particular, let $V$ be a dgmodule. The collection of graded $R$-modules $\End^n(V):=\Hom(V^{\otimes n}, V)$ forms an operad, hence a graded pre-Lie system. Recall that, for $1\leq k\leq n$, the {\sl insertion map at place $k$},
 $\circ_k: \End^n_i(V)\otimes \End^m_j(V)\rightarrow \End^{n+m-1}_{i+j}(V)$ is defined by
$$f\circ_k g=f(\id^{\otimes k-1}\otimes g\otimes \id^{\otimes n-k}).$$
\end{ex}

\begin{defn}
 Let $C\in\dgmod$. A {\sl graded pre-Lie algebra} structure on $C$ is a graded $R$-bilinear map $\circ: C\otimes C\rightarrow C$ satisfying 
\begin{equation}\label{E:preLie}
\forall a,b,c\in C,\quad (a\circ b)\circ c-a\circ(b\circ c)=\epsilon_{b,c}\; a\circ (c\circ b)-\epsilon_{b,c}\; (a\circ c)\circ b.
\end{equation}
\end{defn}

\begin{prop}\label{P:Jacobi}
 Let $C$ be a graded pre-Lie algebra. The bracket defined by
$$\forall c,d\in C, \quad [c,d]=c\circ d-\epsilon_{c,d}\; d\circ c,$$
endows $C$ with a graded Lie algebra structure. Namely, it satisfies the graded antisymmetry and graded Jacobi relations:
\[
[c,d]=-\epsilon_{c,d}\; [d,c], \]
\[\epsilon_{a,c}\; [a,[b,c]]+ \epsilon_{b,a}\; [b,[c,a]]+   \epsilon_{c,b}\; [c,[a,b]]=0.
\]
 \end{prop}

\begin{proof}
 The first equation is immediate. The second one relies on the pre-Lie relation (\ref{E:preLie}):
\begin{multline*}
 \epsilon_{a,c}\; [a,[b,c]]+ \epsilon_{b,a}\; [b,[c,a]]+   \epsilon_{c,b}\; [c,[a,b]]=
\epsilon_{a,c}\;(a\circ (b\circ c)-(a\circ b)\circ c-\epsilon_{b,c}\; a\circ(c\circ b)+
\epsilon_{b,c}\; (a\circ c)\circ b) + \\
\epsilon_{b,a}\; (b\circ (c\circ a)-(b\circ c)\circ a-\epsilon_{a,c}\; b\circ (a\circ c)+\epsilon_{a,c}\; (b\circ a)\circ c) 
+\epsilon_{c,b}\; (c\circ (a\circ b)-(c\circ a)\circ b-\epsilon_{a,b}\; c\circ (b\circ a)+\epsilon_{a,b}\; (c\circ b)\circ a)=0.
\end{multline*}

\end{proof}

\begin{prop}[Gerstenhaber \cite{Ger63}]\label{P:unsignedcirc} Any graded pre-Lie system $(\O,\circ)$ gives rise to a graded pre-Lie algebra
$\O_L=\oplus_n \O^n$ with the pre-Lie product given by
$$\begin{array}{cccc}
\star:& \O^n\otimes \O^m\subset \O_L\otimes \O_L&\rto& \O^{n+m-1}\subset \O_L \\
&f\otimes g & \mapsto & f\star g=\sum_{k=1}^n f\circ_k g
\end{array}$$ 
The associated graded Lie structure is denoted by 
$$\{f,g\}=f\star g-(1)^{ij}g\star f,\  \text{ with } f\in \O^n_i,g\in \O^m_j.$$
\end{prop}


Using Proposition \ref{P:graded2weight}, one gets the following Corollary.

\begin{cor}\label{C:signcirc} Any graded pre-Lie system $(\O,\circ)$ gives rise to a (weight) graded pre-Lie algebra
$$(\O_{wL})_p=\bigoplus\limits_{i,n|i+n-1=p} \O^n_i$$ with the pre-Lie product given by:  $\forall f\in \O^n_i, g\in \O^m_j$, 
$$ f\circ g=(-1)^{|g|(n-1)}\sum_{k=1}^n (-1)^{(m-1)(k-1)} f\circ_k g,$$
with $|g|=m+j-1$.
The associated (weight) graded Lie structure is denoted by 
$$[f,g]=f\circ g-(-1)^{|f||g|} g\circ f.$$
\end{cor}

Next Lemma is a technical lemma that will be useful in the sequel. We will see in the proof, that this lemma is independent of the ring $R$ that we consider and there is no assumption concerning the $2$-torsion of the $R$-modules considered. This is a new fact that can be of independent interest.

\begin{lem}\label{L:preLie_formula} Let $(\cal O,\circ)$ be a graded pre-Lie system.
\begin{itemize}
\item Let $g\in \O$ be an odd degree element. Then $\forall f \in \O$ one has
\begin{equation*}
(f\star g) \star g=f\star (g\star g) \text{ and }
\end{equation*}
\begin{equation*}
 \{f,g\star g\}=-\{g,\{g,f\}\}=-\{g\star g,f\}.
\end{equation*}
\item Let $g\in\O$ be an  element of odd weight, {\it i.e.} $|g|$ is odd. Then $\forall f \in \O$, one has
\begin{equation}\label{E:J1}
(f\circ g) \circ g=f\circ (g\circ g) \text{ and }
\end{equation}
\begin{equation}\label{E:J2}
 [f,g\circ g]=-[g,[g,f]]=-[g\circ g,f].
\end{equation}
\end{itemize}
\end{lem}

\begin{proof}The proof is the same in the two cases. Let us focus on the weight graded case. Let $i$ be the weight of $f$. Relation (\ref{E:J2}) is a consequence of Relation (\ref{E:J1}), for
\begin{multline*}
-[g,[g,f]]=-g\circ (g\circ f)+(-1)^i g\circ (f\circ g)+(-1)^{i+1} (g\circ f)\circ g-(-1)^{i+1+i}(f\circ g)\circ g=\\
(-1)^i(\underbrace{  g\circ (f\circ g)- (g\circ f)\circ g-(-1)^i g\circ (g\circ f)+(-1)^i (g\circ g)\circ f}_{=0 \text { by (\ref{E:preLie})}})
-(g\circ g)\circ f+\underbrace{(f\circ g)\circ g}_{=f\circ (g\circ g)}=[f,g\circ g].
\end{multline*}
Note that Relation (\ref{E:J1}) is a consequence of the pre-Lie relation in case every $R$-module $\cal O_i^n$ has no 2-torsion, for if $g$ has odd weight, then $2(f\circ g) \circ g-2f\circ (g\circ g)=0$.
This is still true without this assumption, if one looks closely at the definition of the (weight) graded pre-Lie structure of Corollary \ref{C:signcirc}. 
The weight graded pre-Lie system relation (\ref{E:syst1}) gives
$$(f\circ g)\circ g =f\circ (g\circ g)+\sum_{u=1}^n\sum_{v=1}^{u-1} (f\circ_u g)\circ_v g+
\sum_{u=1}^n\sum_{v=u+m}^{n+m-1}(f\circ_u g)\circ_v g.$$
The weight graded pre-Lie system relation (\ref{E:syst2}) implies that
$$\sum_{u=1}^n\sum_{v=u+m}^{n+m-1}(f\circ_u g)\circ_v g=\sum_{u=1}^n\sum_{k=u+1}^n(f\circ_u g)\circ_{k+m-1} g=-\sum_{k=1}^n\sum_{u=1}^{k-1} (f\circ_k g)\circ_u g,$$
which ends the proof.
\end{proof}

\subsection{Application to $\End(V)$}

In the sequel we will be concerned with the $(\Z,\N)$-bigraded $R$-module $\End^n_i(V)=\End_i(V^{\otimes n},V)$ where $V$ is
a dgmodule. Example \ref{E:fund}, Proposition \ref{P:unsignedcirc} and Corollary \ref{C:signcirc} assemble in the following Proposition.

\begin{prop}\label{P:end} Let $V$ be a dgmodule. The $(\Z,\N)$-bigraded $R$-module $\End(V)$ forms a graded pre-Lie system.
Consequently, $\forall f\in\End^n(V)$, the product
$$f\star g=\sum_{k=1}^{n} f(\id^{k-1}\otimes g\otimes \id^{n-k})$$
endows $\End(V)$ with a structure of graded pre-Lie algebra and the product
$$f\circ g=(-1)^{|g|(n-1)}\sum_{k=1}^n (-1)^{(m-1)(k-1)} f(\id^{k-1}\otimes g\otimes \id^{n-k})$$ endows $\End(V)$ 
with a structure of  (weight) graded pre-Lie algebra.
\end{prop}

\begin{rem} The signs obtained in the equivalence between graded pre-Lie systems and weight graded pre-Lie systems in Proposition \ref{P:graded2weight} come from a bijection between $\End(V)$ and $\End(sV)$. Let us consider the isomorphism
$\Theta$ of Getzler and Jones in \cite{GetJon90}
\begin{equation}\label{E:Theta}
 \begin{array}{cccc}
  \Theta: & \Hom_i((sV)^{\otimes n},sV)&\rto& \Hom_{i+n-1}((V)^{\otimes n}, V) \\
&F & \mapsto & \Theta(F)
 \end{array}
\end{equation}
defined by
$$s\Theta(F)(s^{-1})^{\otimes n}=F.$$
For $F\in \End^n_i(sV)$ and $G\in \End^m_j(sV)$ one has,
\begin{align*}
s\Theta(F\circ_k G)(s^{-1})^{\otimes n+m-1}&=&F(\id^{\otimes k-1}\otimes G\otimes \id^{n-k})=
s\Theta(F)(s^{-1})^{\otimes n}(\id^{\otimes k-1}\otimes s\Theta(G)(s^{-1})^{\otimes m}\otimes \id^{\otimes n-k})\\
&=&(-1)^{j(n-k)}s\Theta(F)((s^{-1})^{\otimes k-1}\otimes \Theta(G)(s^{-1})^{\otimes m}\otimes (s^{-1})^{\otimes n-k})\\
&=&(-1)^{j(n-k)}(-1)^{(k-1)(j+m-1)}  s\Theta(F)(\id^{\otimes k-1}\otimes \Theta(G)\otimes \id^{n-k})(s^{-1})^{\otimes n+m-1},
\end{align*}
hence $(-1)^{(m-1)(k-1)+|\Theta(G)|(n-1)}\Theta(F)\circ_k\Theta(G)=\Theta(F\circ_k G)$. Consequently,
\begin{equation}\label{E:comparecirc}
\Theta(F)\circ \Theta(G)=\Theta(F\star G).
\end{equation}
\end{rem}

In particular, Lemma \ref{L:preLie_formula} applies for the graded pre-Lie system $\End(V)$. Next proposition states that both pre-Lie products behave well with respect to the differential of the dgmodule $V$.

\begin{prop}\label{P:partialcirc} Let $V$ be a dgmodule with differential $m_1$. The induced differential $\partial$ 
on $\End(V)$ satisfies, $\forall f\in \End^n_i(V), \; \forall g\in \End(V),$

\begin{align*}
 \partial f&=\{m_1, f\}, &\partial f&= [m_1,f];  \\
\partial(f\star g)&=\partial f\star g+(-1)^{i} f\star \partial g & \partial(f\circ g)&=\partial f\circ g+(-1)^{|f|} f\circ \partial g;  \\
\partial\{f,g\}&=\{\partial f,g\}+(-1)^{i} \{f,\partial g\}, &\partial[f,g]&=[\partial f,g]+(-1)^{|f|} [f,\partial g].
\end{align*}
As a consequence $\End(V)$ is a differential  graded Lie algebra and a differential (weight) graded Lie algebra. 
\end{prop}

\begin{proof} The differential $m_1$ is considered as an element of $\End_{-1}^1(V)$, hence of degree $-1$ and of weight $-1$. Recall that $\forall f\in\End^n_i(V),$ one has
$$
\partial f=m_1\circ_1 f-(-1)^i\sum_{k=1}^m f\circ_k m_1= \{m_1,f\}
=(m_1\circ f-(-1)^{i+n-1}f\circ m_1)=[m_1,f].
$$
The proof will be the same for the pre-Lie product $\star$ and the pre-Lie product $\circ$. Let us prove it for $\circ$.
Using the pre-Lie relation (\ref{E:preLie}), one gets
\begin{multline*}
 \partial(f\circ g)-\partial f\circ g-(-1)^{|f|} f\circ \partial g=
m_1\circ (f\circ g)-(-1)^{|f|+|g|}(f\circ g)\circ m_1  
-(m_1\circ f)\circ g+\\
(-1)^{|f|}(f\circ m_1)\circ g -(-1)^{|f|}f\circ (m_1\circ g)+(-1)^{|f|+|g|}f\circ (g\circ m_1)= 
m_1\circ (f\circ g)-(m_1\circ f)\circ g.
\end{multline*}
The last term of the equalities vanish because of Relation (\ref{E:syst1}). The relation
$\partial[f,g]=[\partial f,g]+(-1)^{|f|} [f,\partial g]$ is immediate.
\end{proof}

\begin{rem}\label{R:op} Assume we are given an operad $\mathcal P$ in graded $R$-modules, that is,
a collection $(\mathcal P(n))_{n\geq 1}$ where $\mathcal P(n)$ is a graded $R$-module 
$\mathcal P(n)=\oplus_{i\in\Z} \mathcal P^n_i$. The axioms for the operad are exactly the ones of Proposition \ref{P:unsignedcirc}, where $\circ_k$ denotes the partial composition. As a consequence
$\star$ determines a graded pre-Lie structure on $\mathcal P$ and $\circ$ a weight graded pre-Lie structure on $\mathcal P$ where $f\in\mathcal P^n_i$ has weight $i+n-1$. Lemma \ref{L:preLie_formula} applies also in this case. The same is true for the convolution operad $\Hom(\cal C,\cal P)$ as noticed in the introduction of the section. In particular, the convolution operad forms a graded pre-Lie system.
 \end{rem}

\section{Homology of graded R-modules of homomorphisms} \label{sec:homology}

In this section, we give the conditions on the complexes $C$ and $D$ so that the map
$$H(\Hom(C,D))\rto \Hom(H(C),H(D))$$ is an isomorphism (Proposition \ref{P:lift}) and that the map
$$H(\Hom(C^{\otimes n}, C))\rto \Hom(H(C)^{\otimes n},H(C))$$ is an isomorphism (Corollary \ref{C:lift}).
The last result is one of the key ingredient in order to prove the obstruction Theorem \ref{T:main}. It might be also of independent interest.

\begin{defn}
Let $C$ and $D$ be dgmodules. We denote by $\Hom(C,D)$ the dgmodule
$$\Hom_i(C,D)=\prod_n \Hom_R(V_n,W_{n+i})$$
with differential
$\partial: \Hom_i(C,D)\rightarrow\Hom_{i-1} (C,D)$
defined for $c\in C_n$ by
$$(\partial f)_n(c)=\partial_D(f_n(c))-(-1)^i f_{n-1}(\partial_C c).$$

The graded $R$-module of cycles in $C$ is $Z_i(C)=\Ker(\delta_C: C_i\rto C_{i-1})$  and
$B_i(C)=\im(\delta_C: C_{i+1}\rto C_{i})$ is the graded $R$-module of boundaries in $C$. The homology of $C$ is the graded $R$-module $H_i(C)=Z_i(C)/B_i(C)$.

One has $\partial f=0$ if and only if $f$ is a morphism of differential graded $R$-modules.
In particular $f(Z(C))\subset f(Z(D))$ and $f(B(C))\subset B(D)$. As a consequence, if 
$f\in\Hom_i(C,D)$ and $\partial f=0$, then $f$ defines a map $\bar f\in \Hom_i(H(C),H(D))$ as $\bar f([c])=[f(c)]$. Moreover, if $f=\partial u$, then $f(Z(C))\subset B(D)$ and $\bar f=0$. Thus
one has a well defined map
$${\mathcal H}_{C,D}: H(\Hom(C,D))\rightarrow \Hom(H(C),H(D)).$$

\end{defn}
\begin{defn}
We say that a dgmodule $C$  satisfies {\sl assumption (A)}, if
the sequences  $0\rto Z(C)\rto C \xrightarrow{\partial_C} B(C)\rto 0$ and
 $0\rto B(C)\rto Z(C)\rto H(C)\rto 0$ are split exact.

\end{defn}

\begin{prop}\label{P:lift} Let $C$ and $D$ be dgmodules satisfying assumption (A).

\begin{itemize}
 \item[a)] Given $g\in \Hom_i(H(C),H(D))$, there exists $f\in \Hom_i(C,D)$ such that $\partial f=0$ and $\bar f=g$.
\item[b)] For $f\in \Hom_i(C,D)$ satisfying $\partial f=0$ and $\bar f=0\in \Hom_i(H(C),H(D))$, there exists
$u\in \Hom_{i+1}(C,D)$ such that $\partial u=f$.
\end{itemize}
Consequently the map ${\mathcal H}_{C,D}: H(\Hom(C,D))\rto \Hom(H(C),H(D))$ is an isomorphism of graded $R$-modules and the dgmodule $\Hom(C,D)$ satisfies assumption (A).
\end{prop}

\begin{proof} The short exact sequence $0\rto Z(C)\rto C\xrightarrow{\partial_C} B(C)\rto 0$ splits.
Let $\tau_C:B_{n-1}(C)\rto C_n$ denote a splitting so that $C_n=Z_n(C)\oplus \tau_C(B_{n-1}(C))$.
The short exact sequence $0\rto B(C)\rto Z(C)\rto H(C)\rto 0$ splits.
Let $\sigma:H(C)\rto Z(C)$ denote a splitting so that $Z_n(C)=\sigma(H_n(C))\oplus B_n(C)$. 
Consequently
$$C_n=\sigma(H_n(C))\oplus B_n(C)\oplus \tau_C(B_{n-1}(C))\text{ with } \partial_C(\sigma(h)+ y+\tau_C(z))=z\in  B_{n-1}(C).$$
We use the same notation for $D$.

Part a) of the proposition is proved building $f$ as 
$$\begin{array}{cccc}
 f_n:& C_n=\sigma(H_n(C))\oplus B_{n}(C)\oplus \tau_C(B_{n-1}(C))&\rightarrow&   D_{n+i}=\sigma(H_{n+i}(D))\oplus B_{n+i}(D)\oplus 
 \tau_D(B_{n+i-1}(D)) \\
&  c=\sigma(h)+y+\tau_C(z) & \mapsto& \sigma(g_n(h)) \in  \sigma(H_{n+i}(D)).
\end{array}$$
The equality $\partial_D f_n(c)=0=(-1)^if_{n-1}(\partial_C c)$ implies that $\partial f=0$ and $\overline{f}=g$.

Let us prove part b). Since $\bar f=0$, the map $f$ satisfies $f(\sigma(H(C)))\subset B(D)$.
The map $u\in \Hom_{i+1}(C,D)$ defined by
$$u_n(\sigma(h)+y+\tau_C(z))=(-1)^i f_{n+1}(\tau_C(y))+\tau_Df_n(\sigma(h)),$$
satisfies $\partial u=f$, for
$(\partial u)_n(\sigma(h)+ y+\tau_C(z))=\partial_D\left((-1)^i f_{n+1}(\tau_C(y))+\tau_Df_n(\sigma(h))\right)-(-1)^{i+1}(-1)^if_n(\tau_C z)=
f_n(\partial_C\tau_C y)+f_n(\sigma(h))+f_n(\tau_Cz)$ and $\partial_C\tau_C y=y$.

As a consequence, the map ${\mathcal H}_{C,D}$ is an isomorphism.
The proof of part a) builds an explicit splitting of the projection $Z(\Hom(C,D))\rto
 H(\Hom(C,D))\simeq \Hom(HC,HD)$ while the proof of part b) builds an explicit splitting of the map $\partial:\Hom(C,D)\rto B(\Hom(C,D))$.
Hence, if $C$ and $D$ satisfy assumption (A), so does 
$\Hom(C,D)$.
\end{proof}

\begin{cor}\label{C:lift} Let $C$ be a dgmodule such that $Z(C)$ and $H(C)$ are projective graded $R$-modules. For every $n\geq 1$ the map
$$H(\Hom(C^{\otimes n}, C))\rto \Hom(H(C)^{\otimes n}, H(C))$$
is an isomorphism of graded $R$-modules.
\end{cor}

\begin{proof} Note that if $H(C)$ is projective, then the short exact sequence $0\rto B(C)\rto Z(C) \rto H(C)\rto 0$ splits and $B(C)$ is projective because it is a direct summand of $Z(C)$ which is projective. Consequently, the short exact sequence $0\rto Z(C)\rto C\rto B(C)\rto 0$ splits and $C$ satisfies assumption (A).

The proof is by induction on $n$, applying recursively Proposition  \ref{P:lift}.
For $n=1$, the corollary amounts to the statement of Proposition \ref{P:lift}, with $D=C$. Moreover
$\Hom(C,C)$ satisfies assumption (A).

Because $C$ is a dgmodule such that $Z(C)$ and $H(C)$ are projective, the K\"unneth formula applies (see e.g. \cite{MacLane95}), that is, for every $n$ one has $H(C^{\otimes n})\simeq H(C)^{\otimes n}$.

Let $n>1$. Assume that  the map $H(\Hom(C^{\otimes n-1},C))\rto \Hom(H(C^{\otimes n-1}),H(C))$ is an isomorphism and that $\Hom(C^{\otimes n-1},C)$
satisfies assumption (A). Then the following sequence of maps is a sequence of isomorphisms:
\begin{multline*}
H(\Hom(C^{\otimes n}, C))\simeq H(\Hom(C,\Hom(C^{\otimes n-1},C)))\rto \Hom(H(C),H(\Hom(C^{\otimes n-1},C)))\rto \\
\Hom(H(C),\Hom(H(C^{\otimes n-1}),H(C)))\simeq \Hom(H(C),\Hom(H(C)^{\otimes n-1},H(C)))\simeq \Hom(H(C)^{\otimes n},H(C)))
\end{multline*}
and $\Hom(C^{\otimes n},C)=\Hom(C,\Hom(C^{\otimes n-1},C))$ satisfies assumption (A).
\end{proof}

\section{Obstruction to $A_\infty$-structures} \label{sec:obstruction}

This section is devoted to the obstruction theorem. We first introduce $A_\infty$-algebras, $A_r$-algebras, Hochschild cohomology and prove Theorem \ref{T:main}.

\subsection{$A_\infty$-algebras} \label{ssec:Ainfty}

There are mainly two equivalent definitions of $A_\infty$-algebras.

\begin{defn}\label{D:Ainfty} Let $V$ be a graded $R$-module. We denote by $T^c(sV)$ the free conilpotent coalgebra generated by the suspension of $V$. An {\sl $A_\infty$-algebra structure} on $V$ is a degree $-1$ coderivation $\partial$ on $T^c(sV)$ of square $0$. Namely, the universal property of $T^c(sV)$ implies that $\partial$ is  determined by the sequence $\partial_n: (sV)^{\otimes n}\rto sV$ for $n\geq 1$, obtained as the composite
$$(sV)^{\otimes n}\hookrightarrow T^c(sV)\xrightarrow{\partial} T^c(sV)\epi sV.$$
Conversely, given a sequence $\partial_n\in\Hom_{-1}((sV)^{\otimes n}, sV)$ the unique coderivation on $T^c(sV)$ extending $\partial$ is given by
$$\partial(sv_1\otimes\ldots\otimes sv_n)=\sum_{j=1}^n\sum_{k=1}^{n+1-j} (-1) ^{|sv_1|+\ldots+|sv_{k-1}|} sv_1\otimes\ldots\otimes sv_{k-1}\otimes\partial_j(sv_k\otimes\ldots\otimes sv_{k+j-1})\otimes\ldots\otimes sv_n.$$
\end{defn}

\begin{prop} Let $V$ be a graded $R$-module. The following definitions are equivalent. An $A_\infty$-algebra structure on $V$ is
\begin{itemize}
\item[a)] a collection of elements $\partial_i\in\Hom_{-1}((sV)^{\otimes i},sV), i\geq 1,$ satisfying with the notation of Proposition \ref{P:unsignedcirc}
$$\forall n\geq 1,\ \sum_{i+j=n+1} \partial_i\star \partial_j=0, \text{ or }$$
\item[b)] a collection of elements  $m_i\in\Hom_{i-2}((V)^{\otimes i},V), i\geq 1$ satisfying, with the notation of Theorem \ref{P:end}
\begin{equation}\label{E:Ainfty}
\forall n\geq 1,\ \sum_{i+j=n+1} m_i\circ m_j=\sum_{i+j=n+1} (-1)^{i-1}\sum_{k=1}^i(-1)^{(j-1)(k-1)}m_i\circ_k m_j=0.
\end{equation}
\end{itemize}
\end{prop}

\begin{proof} To prove part a) of the proposition, it is enough to apply Definition \ref{D:Ainfty} and compute

\begin{multline*}
(\partial^2)_n(sv_1\otimes\ldots\otimes sv_n)= \\
\sum_{i+j=n+1}\sum_{k=1}^{n+1-j} (-1) ^{|sv_1|+\ldots+|sv_{k-1}|} \partial_i(sv_1\otimes\ldots\otimes sv_{k-1}\otimes\partial_j(sv_k\otimes\ldots\otimes sv_{k+j-1})\otimes\ldots\otimes sv_n)=\\
\sum_{i+j=n+1}(\partial_i\star \partial_j)(sv_1\otimes\ldots\otimes sv_n).
\end{multline*}

To prove part b) we use the isomorphism $\Theta$ defined in (\ref{E:Theta}), setting $m_i$ to be $\Theta(\partial_i)$.  By relation (\ref{E:comparecirc}), one gets
$$\partial^2=0 \Longleftrightarrow \forall n, \sum_{i+j=n+1} m_i\circ m_j=0.$$
For the definition of $m_i\circ m_j$ we refer to Theorem \ref{P:end}.
\end{proof}

\begin{rem} Note that there exist different sign conventions for the definition of an $A_\infty$-algebra. 
Choosing the bijection
$$\widetilde\Theta: \End(sV)\rightarrow \End(V)$$
defined by $\widetilde\Theta(F)=s^{-1}Fs^{\otimes n}$ and letting $\widetilde m_i=\widetilde\Theta(b_i)$ one gets the original definition of J. Stasheff in \cite{Stasheff63}: the collection of operations $\tilde m_i:A^{\otimes i}\rightarrow A$ of degree $i-2$ satisfies the relation
\begin{equation}\label{E:Ainftygeom}
\forall n\geq 1,\sum_{i+j=n+1} (-1)^{jn}\sum_{k=1}^i(-1)^{k(j-1)}\tilde m_i\circ_k \tilde m_j=0.
\end{equation}
This is equivalent to our definition, because 
$$m_i=(-1)^{\frac{i(i-1)}{2}} \widetilde m_i.$$
\end{rem}

\begin{defn}Let $r>0$ be an integer. A  graded $R$-module $V$ is  an {\sl $A_r$-algebra} if there exists
a collection of elements $m_i\in\Hom_{i-2}(V^{\otimes i},V),$ for $1\leq i\leq r$, such that 
\begin{equation*}
\forall 1\leq n\leq r,\ \sum_{i+j=n+1} m_i\circ m_j=0.
\end{equation*}
\end{defn}

\begin{rem}
Note that $V$ is an $A_1$-algebra if and only if $V$ is a dgmodule, with differential $m_1$ of degree $-1$. 
Recall from Proposition \ref{P:partialcirc} that the induced differential $\partial$ on $\End(V)$ satisfies $\partial f=[m_1,f]$.

The dgmodule $V$ is an $A_2$-algebra  if and only if there exists an element $m_2\in\End^2_0(V)$ such that $\partial m_2=0$, that is, $m_2$ is a morphism of dgmodules.

An $A_3$-algebra is an $A_2$-algebra such that $m_2$ is associative up to homotopy: there exists $m_3\in \End^3_1(V)$ such that $\partial m_3=-m_2\circ m_2.$
Since $m_2$ has odd weight and $\partial m_2=0$ one gets from section \ref{sec:homology} that $m_2$ defines a map $\overline{m_2}\in\End(H(V))$
such that $[\overline{m_2},\overline{m_2}]=\overline{[m_2,m_2]}=2\,\overline{m_2\circ m_2}=0$. Namely the graded $R$-module $H(V)$ is a graded associative algebra.
\end{rem}

\subsection{Hochschild cohomology}

In this section we recall some facts concerning Hochschild cohomology of graded associative algebras. Let $(A,m_2)$ be a graded associative algebra.  Recall that  $m_2\in \End^2_0(A)$ is associative if and only if  $m_2\circ m_2=0$.

\begin{lem} Let $(A,m_2)$ be a graded associative algebra. The map ${\mathrm d}=[m_2,-]: \End^n_i(A)\rto \End^{n+1}_i(A)$  has weight degree $1$  and is a differential. The complex so obtained is the Hochschild cochain complex of $A$ and its cohomology is called the Hochschild cohomology of $A$. Note that the cohomology is bigraded
and is denoted by $HH^n_i(A)$ when the grading needs to be specified.
\end{lem}

\begin{proof}From relation (\ref{E:J2}) one has $\dd^2(f)=[m_2,[m_2,f]]=-[f,m_2\circ m_2]=0$.
\end{proof}

When $A$ is an $A_2$-algebra, the map $\dd$ is still defined but does not satisfy $\dd^2=0$. Nevertheless, we have the following lemma:

\begin{lem}\label{L:commute} Let $A$ be an $A_2$-algebra, with structure maps $m_1$ and $m_2$. The maps $\partial=[m_1,-]:\End^n_i(A)\rto \End^n_{i-1}(A)$ and $\mathrm d=[m_2,-]:\End^n_i(A)\rto \End^{n+1}_i(A)$ satisfy
\[
\begin{cases}
\partial^2=0, \\
\partial \dd=-\dd\partial.
\end{cases}
\]
\end{lem}

\begin{proof} Note that since $m_1$ has weight $-1$ and $m_2$ has weight $1$, they are both elements of odd weight. Hence, equality $\partial^2=0$ is a consequence of $m_1\circ m_1=0$ and relation (\ref{E:J2}).
Let $f$ be an element of weight  $i$  in $\End(A)$. Proposition  \ref{P:Jacobi} and  relation $m_1\circ m_2+m_2\circ m_1=0=[m_1,m_2]$ imply that

$\partial\dd(f)=[m_1,[m_2,f]]=(-1)^{i+1}\left(-[m_2,[f,m_1]]+ (-1)^i[f,[m_1,m_2]]\right)=-[m_2,[m_1,f]]=-\dd\partial(f).$
\end{proof}

\subsection{Obstruction theory} \label{ssec:obstruction}

\begin{thm}\label{T:main}Let $r\geq 3$. Let $A$ be a dgmodule such that $H(A)$ and $Z(A)$ are graded projective $R$-modules. Assume $A$ is an $A_r$-algebra, with structure maps $m_i\in\End^i_{i-2}(A)$ for $1\leq i\leq r$.
The obstruction to lift the $A_{r-1}$-structure of $A$ to an $A_{r+1}$-structure lies in $HH^{r+1}_{r-2}(H(A))$.
\end{thm}

\begin{proof}By assumption, one has
$\forall n\leq r,\ \sum_{i+j=n+1} m_i\circ m_j=0$ which writes
$$\forall n\leq r,\quad \partial m_n=-\sum\limits_{\substack{i+j=n+1,\\ i,j>1}} m_i\circ m_j.$$
The weight of $m_i$ is $i-2+i-1=2i-3$, thus odd.
Let 
$$\mathcal O_{r+1}=\sum\limits_{\substack{i+j=r+2, \\ i,j>1}} m_i\circ m_j \in \Hom_{r-2}(A^{\otimes{r+1}},A).$$
Proposition \ref{P:partialcirc} gives
\[
 \partial \mathcal O_{r+1}=\sum\limits_{\substack{a+b+c=r+3, \\ a,b,c>1}} -(m_a\circ m_b)\circ m_c+m_a\circ (m_b\circ m_c)
\]
The sum splits into the following sums:
\begin{itemize}
 \item[]  If $a,b,c\in\{2,\ldots,r\}$ are distinct integers, one gets the twelve terms of the Jacobi relation, i.e.
$$\sum\limits_{\substack{1< a<b<c\leq r\\ a+b+c=r+3}}[m_a,[m_b,m_c]]+[m_b,[m_c,m_a]]+[m_c,[m_a,m_b]]=0.$$
\item[] Regrouping the terms where $a=b$ and $c\not=a$ or $a=c$ and $b\not=a$, one gets the four terms of the pre-Lie relation of the form, 
\[\sum\limits_{\substack{\alpha\not=\gamma, \alpha,\gamma>1\\ 2\alpha+\gamma=r+3}}
-(m_\alpha\circ m_\alpha)\circ m_\gamma+
m_\alpha\circ (m_\alpha\circ m_\gamma)-(m_\alpha\circ m_\gamma)\circ m_\alpha+ m_\alpha\circ (m_\gamma\circ m_\alpha)=0
\]
\item[] If $b=c\in\{2,\ldots,r\}$, relation (\ref{E:J1}) implies
$$\sum\limits_{\substack {1\leq a, 1<b\leq r\\ a+2b=r+3}}-(m_a\circ m_b)\circ m_b+m_a\circ (m_b\circ m_b)=0.$$
\end{itemize}
Consequently $\partial\mathcal O_{r+1}=0$ and $\mathcal O_{r+1}$ gives rise to an element $\overline{\mathcal{O}_{r+1}}\in \End^{r+1}_{r-2}(H(A))$.
Again, by splitting the sum,
\[
 \dd {\mathcal{O}_{r+1}}=\sum_{a+b=r+2, a,b>1}[m_2,m_a\circ m_b]
\]
and using the relation (\ref{E:J2}) one gets
\begin{itemize}
 \item[] If $a=2$ or $b=2$, then $[m_2,[m_2,m_r]]=[m_2\circ m_2,m_r]=-[\partial m_3,m_r]$;
\item[] If $a\not=b, a,b>2$, then  $[m_2,[m_a,m_b]]=-[m_a,[m_2,m_b]]-[m_b,[m_2,m_a]]$;
\item[] If $a=b, a>2$, then $[m_2,m_a\circ m_a]=-[m_a,[m_2,m_a]).$
\end{itemize}
Thus, on the one hand,
\[
 \dd {\mathcal{O}_{r+1}}=-[\partial m_3,m_r]-\sum_{a+b=r+2, a,b>2}[m_a,[m_2,m_b]].
\]
On the other hand, by splitting the sum and using the computation of $\partial\mathcal O_{r+1}$, one gets that
\begin{align*}
 \partial(\sum\limits_{\substack{ a+b=r+3, \\ a,b>2}} m_a\circ m_b)&=\sum\limits_{\substack{a+b=r+3, \\ a,b>2}} (\partial m_a)\circ m_b-m_a\circ\partial(m_b)\\
&=[\partial m_3,m_r]+\sum\limits_{\substack{a+b=r+2, \\ a,b>2}}- [m_2,m_b]\circ m_a+m_a\circ[m_2,m_b] \\
&=[\partial m_3,m_r]+\sum\limits_{\substack{a+b=r+2, \\ a,b>2}}[m_a,[m_2,m_b]]= -\dd {\mathcal{O}_{r+1}}
\end{align*}
As a consequence $\dd(\overline{\mathcal{O}_{r+1}})=0$. 
If the class of $\overline{\mathcal{O}_{r+1}}$ vanishes  in  $HH^{r+1}_{r-2}(H(A))$, then there exists $u\in \End^{r}_{r-2}(H(A))$ such that $\dd u=\overline{\mathcal{O}_{r+1}}$.

The hypotheses made on the $R$-module $A$ allow us to apply Corollary \ref{C:lift}. There exists 
$m'_r\in \End^r_{r-2}(A)$ such that $\partial m'_r=0$ and $\overline{m'_r}=u$. Moreover
$$\overline{[m_2,m'_r]}=\overline{\dd m'_r}=\dd u=\overline{\mathcal O_{r+1}}=\overline{[m_2,m_r]+\sum\limits_{\substack{i+j=r+2,\\ i,j>2}}m_i\circ m_j}.$$
By Corollary \ref{C:lift}, there exists 
$m_{r+1}\in \End^{r+1}_{r-1}(A)$ such that

$$\partial m_{r+1}=[m_2,m'_r-{m}_r]-\sum\limits_{\substack{i+j=r+2,\\ i,j>2}} m_i\circ m_j.$$

As a consequence the collection $\{m_1,\ldots,m_{r-1},m_r-m'_r,m_{r+1}\}$ is an $A_{r+1}$-structure on $A$ extending its $A_{r-1}$-structure.
\end{proof}

\begin{cor} Let $A$ be an associative algebra up to homotopy such that $H(A)$ and $Z(A)$ are graded projective $R$-modules. If $HH^{r+1}_{r-2}(H(A))=0, \forall r\geq 3$, then there exists an $A_\infty$-structure on $A$ with $m_1$ the differential of $A$ and $m_2$ its product.
\end{cor}


\def\cprime{$'$} \def\cprime{$'$} \def\cprime{$'$} \def\cprime{$'$}
\providecommand{\bysame}{\leavevmode\hbox to3em{\hrulefill}\thinspace}
\providecommand{\MR}{\relax\ifhmode\unskip\space\fi MR }
\providecommand{\MRhref}[2]{%
  \href{http://www.ams.org/mathscinet-getitem?mr=#1}{#2}
}
\providecommand{\href}[2]{#2}

\bigskip


\begin{thebibliography}{10}

\bibitem{BerMoe03}
Clemens Berger and Ieke Moerdijk, \emph{Axiomatic homotopy theory for operads},
  Comment. Math. Helv. \textbf{78} (2003), no.~4, 805--831.

\bibitem{Ger63}
Murray Gerstenhaber, \emph{The cohomology structure of an associative ring},
  Ann. of Maths \textbf{78} (1963), no.~2, 267--288.

\bibitem{GetJon90}
Ezra Getzler and John D.~S. Jones, \emph{{$A\sb \infty$}-algebras and the
  cyclic bar complex}, Illinois J. Math. \textbf{34} (1990), no.~2, 256--283.

\bibitem{Kadei80}
T.~V. Kadei{\v{s}}vili, \emph{On the theory of homology of fiber spaces},
  Uspekhi Mat. Nauk \textbf{35} (1980), no.~3(213), 183--188, International
  Topology Conference (Moscow State Univ., Moscow, 1979).

\bibitem{KapMan01}
Michael~M. Kapranov and Yuri Manin, \emph{Modules and {M}orita theorem for
  operads}, Amer. J. Math. \textbf{123} (2001), no.~5, 811--838.

\bibitem{Lefevre03}
Kenji Lef\`evre-Hasegawa, \emph{Sur les $a_\infty$-cat\'egories}, PhD thesis
  Universit\'e Paris 7, 2003.

\bibitem{LodVal}
Jean-Louis Loday and Bruno Vallette, \emph{Algebraic operad}, Book in
  preparation, 2010.

\bibitem{MacLane95}
Saunders Mac~Lane, \emph{Homology}, Classics in Mathematics, Springer-Verlag,
  Berlin, 1995, Reprint of the 1975 edition. \MR{1344215 (96d:18001)}

\bibitem{Rob89}
Alan Robinson, \emph{Obstruction theory and the strict associativity of
  {M}orava {$K$}-theories}, Advances in homotopy theory ({C}ortona, 1988),
  London Math. Soc. Lecture Note Ser., vol. 139, Cambridge Univ. Press,
  Cambridge, 1989, pp.~143--152.

\bibitem{Shipley07}
Brooke Shipley, \emph{{$H\Bbb Z$}-algebra spectra are differential graded
  algebras}, Amer. J. Math. \textbf{129} (2007), no.~2, 351--379.

\bibitem{Stasheff63}
James~Dillon Stasheff, \emph{Homotopy associativity of {$H$}-spaces. {I},
  {II}}, Trans. Amer. Math. Soc. 108 (1963), 275-292; ibid. \textbf{108}
  (1963), 293--312.

\end{thebibliography}
\end{document}